\documentclass[12pt]{ amsart}
\usepackage{amssymb}
\usepackage{amsfonts}
\usepackage[all,tips]{xy}

\newtheorem{theorem}{Theorem}
\theoremstyle{plain}

\newtheorem{corollary}{Corollary}

\newtheorem{lemma}{Lemma}

\newtheorem{proposition}{Proposition}

\numberwithin{equation}{section}
\newcommand{\bg}{\begin{gathered}}
\newcommand{\eg}{\end{gathered}}
\newcommand{\bea}{\begin{eqnarray}}
\newcommand{\eea}{\end{eqnarray}}

\begin{document}
\title[On the Mean Summability of \linebreak Series by Nonlinear Fourier Basis  ]{On the Mean Summability of \linebreak Series by Nonlinear Fourier Basis  }
\author{ Hatice  ASLAN \and Ali GUVEN}
\address{Department of Mathematics, Firat University, Faculty of Science, 23119, Elazig, Turkey}
\email{haticeaslan2017@gmail.com (Corresponding author)}
\address{Department of Mathematics
Faculty of Art and Science
Balikesir University
10145, Balikesir
Turkey  }
\email{aliguven@gmail.com}

\begin{abstract}
The nonlinear signal processing has achieved a rapid process in the recent years. A family of nonlinear Fourier bases, as a typical family
of mono-component signals, has been \linebreak constructed and applied to signal processing. In this paper, \linebreak the approximation properties of the
partial sums and Ces\`{a}ro \linebreak summability of  series by the nonlinear Fourier basis are \linebreak investigated in the $L^{p}(\mathbb{T})$.
Furthermore, these results are applied to the prove of  Bernstein's  inequalities for  nonlinear trigonometric polynomials.
\end{abstract}

\maketitle

%\author{Hatice Aslan}

\section{Introduction}

Let $0<p<\infty$ and $\mathbb{T}:= \mathbb{R} / 2 \pi \mathbb{Z}$. Let  $f$ is a  periodic function on $\mathbb{T}$,
then we denote $L^{p}(\mathbb{T})$ the set of Lebesgue measurable  functions $f:\mathbb{T} \rightarrow \mathbb{R}$ (or  $\mathbb{C})$) such that
\begin{equation*}
 \left(\frac{1}{2\pi}\int_{\mathbb{T}}\mid f(x) \mid^{p}dx\right)^{1/p} < \infty
\end{equation*}
where the integral is a Lebesgue integral, and we identify functions that differ on a a set of measure zero. We define $L^{p}$-norm of $f$ by
\begin{equation*}
\parallel f \parallel_{p}= \left(\frac{1}{2\pi} \int_{\mathbb{T}}\mid f(x) \mid^{p}dx\right)^{1/p}.
\end{equation*}
For $p=\infty$ the space $L^{\infty}(\mathbb{T})$ consists of the Lebesgue measurable \linebreak functions $f:\mathbb{T} \rightarrow \mathbb{R}$ (or $ \mathbb{C})$ that  are essentially bounded on $\mathbb{T}$,\linebreak meaning that $f$ is bounded on a subset of $\mathbb{T}$ whose complement has measure zero.  The norm on is essential supremum
\begin{equation*}
\parallel f\parallel_{\infty}= \inf\{ M \mid \quad  \mid f(x)\mid \leq M \quad a.e. \quad in \quad \mathbb{T} \}.
\end{equation*}
Note that $\parallel f \parallel_{p}$ may take the value $\infty$.\\

The another important concept is the modulus of smoothness is \linebreak defined by
\begin{equation*}
\omega \left (f;t \right)_{p}:= \sup_{0 \leq h \leq t}\parallel f(.+h)-f(.) \parallel_{p}
\end{equation*}
If $p=\infty$, then  it is called modulus of contiunity. And this \linebreak nondecreasing continuous function on the interval $[0,2\pi]$ having \linebreak properties:
\begin{equation*}
\omega(0)=0,  \quad \omega(t_{1}+t_{2})\leq \omega(t_{1})+ \omega(t_{2}).
\end{equation*}
A family of nonlinear Fourier bases as the extension of the classical Fourier basis, have been constructed and applied to signal processing \linebreak
\cite{2Chen,3Qian,10Qian,11Qian,12Qian}. For any complex number $a=\left\vert a\right\vert e^{it_{a}},\left\vert a\right\vert <1$, the \linebreak nonlinear phase function $\theta _{a}(t)$ is defined by the radical boundary value of the M\"{o}buis transformation
\begin{equation*}
\tau _{a}=\frac{z-a}{1-\overline{a}z}
\end{equation*}
that is,
\begin{equation*}
e^{i\theta _{a}(t)}:=\tau _{a}(e^{it})=\frac{e^{it}-a}{1-\overline{a}e^{it}}
\end{equation*}
It is easily seen that
\begin{equation}\label{1:1}
\theta _{a}(t+2\pi )=\theta _{a}(t)+2\pi
\end{equation}
and its derivative is the Poisson kernel
\begin{equation*}
\theta _{a}^{\prime }(t)=p_{a}(t)=\frac{1-\left\vert a\right\vert ^{2}}{1-2\left\vert a\right\vert \cos (t-t_{a})+\left\vert a\right\vert ^{2}}
\end{equation*}
which satisfies
\begin{equation}\label{1:2}
0<\frac{1-\left\vert a\right\vert }{1+\left\vert a\right\vert }\leq \theta _{a}^{\prime }(t)\leq \frac{1+\left\vert a\right\vert }{1-\left\vert
a\right\vert }
\end{equation}
Hence, $\theta _{a}(t)$ is a strictly monotonic increasing function, which makes $cos \theta _{a}(t)$ be a special mono-component signal \cite{10Qian,12Qian}. It
has been shown that for any sequence $ \left\{c_{k}\right\}_{k\in \mathbb{Z}}$ of finite nonzero terms, there holds
\begin{equation*}
 \sum_{k \in \mathbb{Z}} \mid c_{k}\mid^{2} =\frac{1}{2\pi} \int_{\mathbb{T}} \sum_{k \in \mathbb{Z}} \mid c_{k} e^{ik\theta_{a}(x)}\mid^{2}dx=\frac{1}{2\pi} \int_{\mathbb{T}}\sum_{k \in \mathbb{Z}} \mid c_{k} e^{ik \theta_{a}(t)}\mid^{2}p_{a}(t) dt
\end{equation*}
which combining with \eqref{1:2} implies that the so-called nonlinear Fourier basis $\left\{ e^{in \theta_{a}(t)}\right\}_{n\in \mathbb{Z}}$ forms a
Riesz basis for $L^{2}(\mathbb{T})$ with the upper bound $\sqrt{\frac{1+\left\vert a\right\vert }{1-\left\vert a\right\vert }}$ and the lower bound
$\sqrt{\frac{1-\left\vert a\right\vert }{1+\left\vert a\right\vert }}$. When $a = 0$,  $\left\{ e^{in \theta_{a}(t)}\right\}_{n\in \mathbb{Z}}$ is \linebreak simply the Fourier basis $\left\{ e^{int}\right\}_{n\in \mathbb{Z}}$.\\

Let $\tau^{a}_{n}$ be the space of all the nonlinear trigonometric polynomials of degree less than or equal to $n$, that is,
\begin{equation*}
\tau_{n}^{a}:= span \left\{e^{ik \theta_{a}(t)}: \mid k \mid \leq n \right\}.
\end{equation*}
The approximation error of $f \in L^{p}(\mathbb{T})$,
\begin{equation*}
E_{n}^{a}(f)_{p} = \inf _{T \in \tau^{a}_{n}}  \parallel f- T\parallel _{L^{p}(\mathbb{T})}.
\end{equation*}
Let us recall some known lemmas (see \cite{Huang}) which will be used in the sequel of paper.
\begin{lemma}
 Let $f\in C(\mathbb{T})$. We have
\begin{equation*}
E_{n}^{a}(f)_{\infty}\leq \frac{24}{1-\mid a \mid}\omega \left(f,\frac{1}{n}\right)_{\infty}.
\end{equation*}
\end{lemma}
\begin{lemma}
 Let $f\in C(\mathbb{T})$. We have
\begin{equation*}
\frac{1-\mid a \mid}{2}\omega(f,t)_{\infty}\leq \omega(f\circ \theta ^{-1}_{a},t)_{\infty} \leq \frac{2}{1-\mid a \mid}\omega(f,t)_{\infty}.
\end{equation*}
\end{lemma}
In the present paper first we deal with some  properties of \linebreak  nonlinear Fourier series. Then  we discuss $n$-th partial sums
and Ces\`{a}ro  sum of nonlinear Fourier series. Also we prove the necessary and \linebreak sufficient condition for nonlinear
Fourier series which governs the $(C,1)$ \linebreak summability in $L^{p}(\mathbb{T})$ for arbitrary function $f$ from
$L^{p}(\mathbb{T})$. This \linebreak result is applied to the prove of Bernstein's inequality for nonlinear  trigonometric polynomials.

\section{Convergence of Nonlinear Fourier Series}

Let $f \in L^{1}(\mathbb{T})$. Then $1\leq p \leq \infty$ by nonlinear basis denoted by

\begin{equation*}
f(x)\thicksim \sum_{k= -\infty}^{\infty} c_{k} e^{ik \theta _{a}(x)}
\end{equation*}
be its series by nonlinear Fourier basis where

\begin{equation*}
c_{k} = c_{k} (f) = \frac{1}{2\pi} \int_{-\pi}^{\pi} f(t) e^{-ik \theta_{a}(t)} p_{a}(t) dt, \quad k \in\mathbb{Z}.
\end{equation*}
For simplicity throughtout the present paper  we write nonlinear Fourier series  as series by nonlinear Fourier basis.
Now we can begin to give some properties of nonlinear Fourier series:

Assume  $f,g \in L^{1}(\mathbb{T})$ and
\begin{equation*}
f(x)\thicksim \sum_{k= -\infty}^{\infty} c_{k} e^{ik \theta _{a}(x)}, \quad
g(x)\thicksim \sum_{k= -\infty}^{\infty} d_{k} e^{ik \theta _{a}(x)}.
\end{equation*}

For linearity  and convolution we have the following properties \linebreak respectively.
\begin{equation*}
(Af+Bg)(x)\thicksim \sum_{k= -\infty}^{\infty} (Ac_{k}+Bd_{k}) e^{ik \theta _{a}(x)}
\end{equation*}

\begin{equation*}
(f*g)(x)\thicksim \sum_{k= -\infty}^{\infty} \left(\frac{1+\mid a \mid}{1-\mid a \mid}\right)(c_{k} d_{k}) e^{ik \theta _{a}(x)}
\end{equation*}
Furthermore remember Sobolev space i.e.
\begin{equation*}
W^{p}_{r}(\mathbb{T})=\left\{f\in X^{p}(\mathbb{T}): f,...,f^{(r-1)}\in \mathcal{AC}(\mathbb{T}), f^{(r)}\in X^{p}(\mathbb{T}) \right\}.
\end{equation*}

Assume that $f \in W^{p}_{1}(\mathbb{T})$.  It can be easily seen that property i.e.
\begin{equation*}
f'(x)\thicksim \sum_{k= -\infty}^{\infty} \left(\frac{1+\mid a \mid}{1-\mid a \mid}\right)(ik c_{k}) e^{ik \theta _{a}(x)}
\end{equation*}
holds.

We wish to examine the convergence of nonlinear Fourier series. To discuss the convergence, pointwise or uniform, of nonlinear Fourier \linebreak series, we need to discuss the convergence of the sequence $\{S^{a}_{n}\}$ of partial sums. We have
\begin{equation*}
S_{n}^{a}(f)(x)= \sum_{k= -n}^{n} c_{k} e^{ik \theta _{a}(x)}
\end{equation*}
where
\begin{equation*}
c_{k} = c_{k} (f) = \frac{1}{2\pi} \int_{-\pi}^{\pi} f(t) e^{-ik \theta_{a}(t)} p_{a}(t) dt.
\end{equation*}

\begin{proposition}
 Let $S_{n}^{a}(f)$ be the sequence of partial sums of the \linebreak nonlinear Fourier series of $f$. Let  $f \in L^{1}(\mathbb{T})$  which is $2\pi$-periodic.
 Then
\begin{equation}\label{2:1}
S_{n}^{a}(f)(x)= \frac{1}{2\pi} \int_{-\pi}^{\pi} f(t) D_{n}(\theta_{a}(x)-\theta_{a}(t)) p_{a}(t)dt.
\end{equation}
\begin{equation}\label{2:2}
= \frac{1}{2\pi} \int_{-\pi}^{\pi} F(\theta_{a}(x)+t) D_{n}(t)dt,
\end{equation}
\begin{equation}\label{2:3}
= \frac{1}{2\pi} \int_{0}^{\pi} \left\{ F(\theta_{a}(x)-t)+F(\theta_{a}(x)+t)\right \}  D_{n}(t)dt.
\end{equation}
\end{proposition}

\begin{proof}
 By the expression for the $c_{k}$  and considering the equality \linebreak $S_{n}^{a}(f)(x)= S_{n}(F)(\theta _{a}(x))$ where $F= f\circ \theta _{a}^{-1}$  we have already established \eqref{2:1}.  By a change of variable $u-\theta_{a}(x)=y$ and the \eqref{1:1} equality of   $\theta_{a}(x)$ phase function's $2\pi$-periodicity,  it follows that the integral does not change as long as the length of the interval of integration is $2\pi$ we get
\begin{equation*}
S_{n}^{a}(f)(x)= \frac{1}{2\pi} \int_{-\pi}^{\pi} F(\theta_{a}(x)+y) D_{n}(y)dy.
\end{equation*}
This proves \eqref{2:2}. Finally, we split the integral in \eqref{2:2} as the sum of integrals over $[-\pi, 0]$ and $[0, \pi]$. Now
\begin{equation*}
\int_{-\pi}^{0} F(\theta_{a}(x)+t) D_{n}(t)dt=\int_{0}^{\pi} F(\theta_{a}(x)+t) D_{n}(t)dt,
\end{equation*}
using the change of variable $y = -t$ and the evenness of $D_{n}$.  This proves \eqref{2:3}.
\end{proof}
 We can begin with properties of the operators $S^{a}_{n}$ of partial sums of nonlinear Fourier series. For this first we need following lemmas.

\begin{lemma}

Let $F:= f\circ \theta_{a}^{-1}$. Then we have

\begin{equation*}
\parallel F \parallel_{p} \leq \left\{
                                                 \begin{array}{ll}
                                                 \parallel f \parallel_{p} , & \hbox{$ p= 1, \infty, f \in  C(\mathbb{T}) \quad if  \quad p= \infty$;} \\
                                                   \left(\frac{1+\left\vert a\right\vert }{1-\left\vert a\right\vert }\right)^{1/p}  \parallel f \parallel_{p} & \hbox{$ 0<p<\infty$;}
                                                 \end{array}
                                               \right.
\end{equation*}

\end{lemma}
\begin{proof}
 Let $f \in \mathcal{C}(\mathbb{T})$ ve $F:= f\circ \theta_{a}^{-1}$. Therefore for  $p= 1, \infty$ we have
\begin{equation*}
\parallel  f \parallel _{\infty}:= sup_{x\in \mathbb{T}} \mid f(x)\mid= sup_{\theta _{a}(x)\in \mathbb{T}}\mid  F(\theta_{a}(x))\mid= \| F\|_{\infty}.
\end{equation*}
And for $0<p<\infty$ the following inequality holds.
\begin{equation*}
\parallel f \parallel_{p}:= \left(\frac{1}{2\pi}\int_{\mathbb{T}}\mid f(t)\mid^{p}dt\right)^{1/p}.
\end{equation*}
By change of variable   $\theta_{a}(t)=u$ and using  phase function's property which is giving in \eqref{1:2}, we have the following equality.

\begin{equation*}
\parallel f \parallel_{p}  \geq \left(\frac{1-\left\vert a\right\vert }{1+\left\vert a\right\vert }\right)^{1/p}\left(\frac{1}{2\pi}\int_{\mathbb{T}}\mid
F(u)\mid^{p}du\right)^{1/p}=\left(\frac{1-\left\vert a\right\vert }{1+\left\vert a\right\vert }\right)^{1/p}\|F \|_{p}.
\end{equation*}
Therefore we have result that we wanted.
\end{proof}

For the spaces  $L^{1}(\mathbb{T})$, $C(\mathbb{T})$, one can evaluate the norms $\parallel S_{n}^{a} \parallel$ by direct computation.

\begin{theorem}
One has
\begin{equation*}
\parallel S_{n}^{a}(f)\parallel_{\infty} \leq \Lambda_{n}\parallel f \parallel_{\infty}.
\end{equation*}
\end{theorem}

\begin{proof}
Let  consider $S_{n}^{a}$ be an operator $C(\mathbb{T})$ to $C(\mathbb{T})$. By using Lemma 2.1 and  equality \eqref{2:2}, we see that each of the
norms \eqref{2:2} is equal to
\begin{equation*}
\mid S_{n}^{a}(f)(x)\mid=\left| \frac{1}{2\pi} \int_{-\pi}^{\pi} F(\theta _{a}(x)+t)D_{n}(t )dt\right|
\end{equation*}
\begin{equation*}
\leq \frac{1}{2\pi} \int_{-\pi}^{\pi} \mid F(\theta _{a}(x)+t)\mid \mid D_{n}(t)\mid dt
\end{equation*}
\begin{equation*}
\leq \frac{1}{2\pi} \int_{-\pi}^{\pi} \parallel F\parallel_{\infty} \mid D_{n}(t)\mid dt
= \frac{1}{2\pi} \int_{-\pi}^{\pi} \parallel f \parallel_{\infty} \mid D_{n}(t) \mid  dt
\end{equation*}
From the Lebesgue constant's definition (see in \cite{Giuseppe})
\begin{equation*}
\parallel S_{n}^{a}(f)\parallel_{\infty} \leq \Lambda_{n}\parallel f \parallel_{\infty}
\end{equation*}
holds.
\end{proof}

\begin{theorem}
For all $f\in C(\mathbb{T})$,
\begin{equation*}
\parallel f-S_{n}^{a}(f)\parallel_{\infty} \leq c \log n E^{a}_{n}(f)
\end{equation*}
holds.
\end{theorem}
\begin{proof}
Let $f\in C(\mathbb{T})$. Consider  linear operator $S_{n}^{a}: \mathcal{C}(\mathbb{T})\rightarrow \mathcal{C}(\mathbb{T})$, $n=0, 1, ...$. The best approximation by $t \in \tau_{n}^{a}$ is
\begin{equation*}
E^{a}_{n}(f)= \parallel f-t\parallel_{\infty}.
\end{equation*}
Thus we can write the following equality.
\begin{equation*}
\parallel f-S_{n}^{a}(f)\parallel_{\infty}=\parallel f-t+t-S_{n}^{a}(f)\parallel_{\infty}
=\parallel f-t+S_{n}^{a}(f)(t-f)\parallel_{\infty}
\end{equation*}
\begin{equation*}
\leq \parallel f-t\parallel_{\infty}+\parallel S_{n}^{a}(f)(t-f)\parallel_{\infty} = E^{a}_{n}(f) +\parallel S_{n}^{a}(f)\parallel_{\infty}  \parallel f-
t\parallel_{\infty}
\end{equation*}
holds. Therefore from Theorem 2.1
\begin{equation*}
\parallel f-S_{n}^{a}(f)\parallel_{\infty} \leq (1+\Lambda_{n})E^{a}_{n}(f)
\end{equation*}
holds. Hence by using theorem for  Lebesgue constant (see in \cite{Giuseppe}) we have
\begin{equation*}
\parallel S_{n}^{a}(f)-f\parallel_{\infty}\leq  c  \log n E^{a}_{n}(f).
\end{equation*}
This completes the proof.
\end{proof}

\begin{theorem}
Let  consider  $S_{n}^{a}: L^{p}(\mathbb{T})\rightarrow L^{p}(\mathbb{T})$ linear operator \linebreak sequence, where $ n=0, 1, ...$ and $1\leq p<\infty$. Then
\begin{equation*}
\parallel S_{n}^{a}(f)\parallel_{p} \leq \left(\frac{1+\left\vert a\right\vert }{1-\left\vert a\right\vert }\right)^{2/p} \Lambda_{n} \parallel f \parallel_{p}
\end{equation*}
holds.
\end{theorem}

\begin{proof}
Let $1\leq p<\infty$. Consider operator $S_{n}^{a}: L^{p}(\mathbb{T})\rightarrow L^{p}(\mathbb{T})$, \linebreak $n=0, 1, ...$. Thus if we consider the inequality \eqref{2:2}
\begin{equation*}
\parallel S_{n}^{a}(f)\parallel_{p}= \left(\frac{1}{2\pi}\int_{-\pi}^{\pi} \mid S_{n}^{a}(f)(x)\mid^{p} dx \right)^{1/p}
\end{equation*}
\begin{equation*}
=\left( \frac{1}{2\pi} \int_{-\pi}^{\pi} \left| \frac{1}{2\pi} \int_{-\pi}^{\pi} F(\theta _{a}(x)-t) D_{n}(t) dt \right| ^{p}dx \right)^{1/p}
\end{equation*}
holds. Hence by using Minkowski integration  inequality (see in \cite{Zygmund}), we have
\begin{equation*}
\parallel S_{n}^{a}(f)\parallel_{p}\leq \frac{1}{2\pi} \int_{-\pi}^{\pi}\left ( \frac{1}{2\pi} \int_{-\pi}^{\pi} \mid F(\theta _{a}(x)+ t)\mid^{p}
p_{a}(x)dx \right) ^{\frac{1}{p}}
\end{equation*}
\begin{equation*}
\times \mid D_{n}(t)\mid \left(\frac{1}{p_{a}(x)}\right)^{1/p}dt
\end{equation*}
\begin{equation*}
\leq \left(\frac{1+\left\vert a\right\vert }{1-\left\vert a\right\vert }\right)^{1/p}\frac{1}{2\pi} \int_{-\pi}^{\pi}(\frac{1}{2\pi}
\int_{-\pi}^{\pi} \mid F(\theta _{a}(x)+t)\mid^{p} p_{a}(x)dx ) ^{\frac{1}{p}} \mid D_{n}(t)\mid dt
\end{equation*}
\begin{equation*}
\leq \left(\frac{1+\left\vert a\right\vert }{1-\left\vert a\right\vert }\right)^{1/p}\frac{1}{2\pi} \int_{-\pi}^{\pi}\parallel F \parallel_{p} \mid
D_{n}(t)\mid dt.
\end{equation*}
Therefore considering Lemma 2.2 and  Dirichlet kernel's definition (see in \cite{Lorentz}) is giving  result that we wanted.
\end{proof}
Now let examine the  convergence of partial sum of the nonlinear Fourier series in $L^{p}(\mathbb{T})$ space.

\begin{theorem}
Let $1\leq p<\infty$. Then for all  $ f \in L^{p}(\mathbb{T})$
\begin{equation}\label{2:4}
\parallel f-S_{n}^{a}(f)\parallel_{p} \rightarrow 0,\quad  n\rightarrow\infty
\end{equation}
holds if and only if, there exist a  $M>0$ constant that only depend on $p$ such that
\begin{equation}\label{2:5}
\parallel S_{n}^{a}(f)\parallel_{p} \leq M_{p}^{a} \parallel f \parallel_{p}.
\end{equation}
\end{theorem}
\begin{proof}
For necessity let  $ f \in L^{p}(\mathbb{T})$ and
$\parallel f-S_{n}^{a}(f)\parallel_{p} \rightarrow 0$, $n\rightarrow\infty$. Thus
$\left\{ S_{n}^{a}(f)\right\}$ is bounded since it  converges in $L^{p}(\mathbb{T})$ norm space for   $ f \in L^{p}(\mathbb{T})$. Therefore there  exists $ M_{f} >0$ for all  $ f \in L^{p}(\mathbb{T})$. Such that $\parallel S_{n}^{a}(f)\parallel_{p}< M_{f}$. Thus considering uniform  bounded principle (e.g. \cite{Simon}) we have
\begin{equation*}
\sup\left\{ \parallel S_{n}^{a}\parallel:n=0, 1,2,...\right\}<\infty.
\end{equation*}
If  we consider Theorem 2.3 and choose
\begin{equation*}
M= \sup \left\{ \parallel S_{n}^{a}\parallel:n=0, 1,2,...\right\}<\infty
\end{equation*}
for  $ f \in L^{p}(\mathbb{T})$, then  we can write for \eqref{2:5} inequality.\\

For sufficiency let $\parallel S_{n}^{a}(f)\parallel_{p} \leq M_{p}^{a} \parallel f \parallel_{p}$. For
\begin{equation*}
E_{n}^{a}(f)_{p}= \parallel f -t_{n}^{a}\parallel_{p}
\end{equation*}
where $ f \in L^{p}(\mathbb{T})$  and $n=0, 1,2,...$ . Therefore we  can prove \eqref{2:4} with the help of the nonlinear polynomials $t_{n}^{a}$ of best approximation of the function:
\begin{equation*}
\parallel f-S_{n}^{a}(f)\parallel_{p} =\parallel f -t_{n}^{a}+t_{n}^{a}+S_{n}^{a}(f)(x) \parallel_{p}
\end{equation*}
\begin{equation*}
\leq \parallel f -t_{n}^{a}\parallel_{p}+\parallel S_{n}^{a}(f-t_{n}^{a})\parallel_{p}\leq (M_{p}^{a} +1) E_{n}^{a}(f)_{p}.
\end{equation*}
Since
\begin{equation*}
E_{n}^{a}(f)_{p}\rightarrow 0,\quad n\rightarrow \infty,
\end{equation*}
this completes the proof.
\end{proof}

\begin{corollary}
  The norms of the operators $S_{n}^{a}(f)$ are bounded in each space $f \in L^{p}(\mathbb{T})$, $1<p< \infty$.
\end{corollary}

 From Theorem 2.2, Theorem 2.3  and Theorem 2.4 we now derive
\begin{equation*}
\parallel S_{n}^{a}(f)-f \parallel_{p} \leq \left\{
                                                 \begin{array}{ll}
                                                 C \log n E^{a}_{n}(f)_{p} , & \hbox{$ p= 1, \infty, f \in  C(\mathbb{T}),if  p= \infty$;} \\
                                                  C_{p} \left(\frac{1+\left\vert a\right\vert }{1-\left\vert a\right\vert }\right)^{2/p}  E^{a}_{n}(f)_{p} , & \hbox{$1<p< \infty$.}
                                                 \end{array}
                                               \right.
\end{equation*}

We see that the partial sums $S_{n}^{a}$  approximate almost as well as its \linebreak polynomial of best approximation. This is true even for $f \in  C(\mathbb{T}) $, if the factor $\log n$ is not essential for the problem considered. \\

For the partial sums $S_{n}^{a}$  of the Fourier series of $f$ we do not have $S_{n}^{a}(f)\rightarrow  f$ in $C(\mathbb{T})$ for each $f$. But we do have fast \linebreak convergence  of $S_{n}^{a}$  for smooth functions  $f$. Actually, the convergence \linebreak $\parallel S_{n}^{a}(f)-f \parallel_{\infty}    \rightarrow  0$ can be arbitrarily fast for some $f$ (that are not trigonometric polynomials): The following theorems shows that it is sufficient to take  $f(t) =  \sum_{k=0}^{\infty} [a_{k} cos(k\theta _{a}(x))+ b_{k}sin(k\theta _{a}(x))]$, where  $a_{k}$, $b_{k} $ converge to zero sufficiently fast without being zero.

 \begin{theorem}
For $f \in L^{1}(\mathbb{T})$
\begin{equation*}
\lim_{n\rightarrow\infty} c^{a}_{n}(f)= \lim_{n\rightarrow -\infty} c^{a}_{n}(f)= 0
\end{equation*}
holds.
 \end{theorem}
\begin{proof}

Let $\varepsilon>0$. In this case there exists a  $t \in \tau^{a}_{n}$ such that

\begin{equation*}
\parallel f-t \parallel< \varepsilon 2\pi.
\end{equation*}
In  this inequality let $t$ has degree  $N$. Therefore we can write

\begin{equation*}
t(x)= \sum^{m=N}_{m=-N} d_{m} e^{im\theta_{a}(x)}
\end{equation*}
Thus  for $\mid k \mid >N$,

\begin{equation*}
c^{a}_{k}(f)= \frac{1}{2\pi} \int^{\pi}_{-\pi} f(x)e^{-ik\theta_{a}(x)}dx
\end{equation*}

\begin{equation*}
= \frac{1}{2\pi} \int^{\pi}_{-\pi} f(x)e^{-ik\theta_{a}(x)}dx- \frac{1}{2\pi} \int^{\pi}_{-\pi} t(x)e^{-ik\theta_{a}(x)}dx
\end{equation*}

\begin{equation*}
= \frac{1}{2\pi} \int^{\pi}_{-\pi} [ f(x)-t(x)] e^{-ik\theta_{a}(x)}dx
\end{equation*}
holds. Hence we obtain the result from  following inequality.

\begin{equation*}
\mid c^{a}_{k}(f)\mid \leq  \parallel f(x)-t(x)\parallel_{1}<\varepsilon.
\end{equation*}
\end{proof}

\begin{corollary}
Let $f \in L^{1}(\mathbb{T})$. Then for all $f$,

\begin{equation*}
\lim_{n\rightarrow\infty} a^{a}_{n}(f)= \lim_{n\rightarrow -\infty} b^{a}_{n}(f)= 0
\end{equation*}
holds.
\end{corollary}
\begin{proof}
By writing $a^{a}_{n}(f)= c^{a}_{n}(f)+c^{a}_{-n}(f)$ ve $b^{a}_{n}(f)=  \frac{c^{a}_{n}(f)+c^{a}_{-n}(f)}{i}$ in \linebreak Theorem 2.5, we can complete the proof.
\end{proof}
\begin{theorem}
Let $f \in L^{1}(\mathbb{T})$. If

\begin{equation*}
 \int_{0}^{\pi} \left|\frac{F(\theta_{a}(x)+t)+F(\theta_{a}(x)-t)-2f(x)}{t}\right| dt<\infty
\end{equation*}
then for  $n\rightarrow\infty$, $S^{a}_{n}(f)(x)\rightarrow f(x)$ holds.
\end{theorem}
\begin{proof}
Let assume  $f \in L^{1}(\mathbb{T})$. By using  \eqref{2:3} equality, we can write

\begin{equation*}
S^{a}_{n}(f)(x)-f(x)= \frac{1}{2\pi} \int_{0}^{\pi} \left\{ F(\theta_{a}(x)+t) + F(\theta_{a}(x)-t) -2 f(x)\right\} D_{n}(t)dt.
\end{equation*}
Therefore considering Dirichlet kernel's property (see in \cite{Zygmund}), we obtain the following.
\begin{equation*}
S^{a}_{n}(f)(x)-f(x)= \frac{1}{2\pi} \int_{0}^{\pi} \left\{ \frac{F(\theta_{a}(x)+t) + F(\theta_{a}(x)-t) -2 f(x)}{\sin \left( \frac{t}{2}\right)}\right\}
\end{equation*}

\begin{equation*}
\times\sin\left(\left(\frac{2n+1}{2}\right)t\right)dt
\end{equation*}
Now let  think the  $\phi^{a}_{x}(t)$ as following.

\begin{equation*}
\phi^{a}_{x}(t)= F(\theta_{a}(x)+t) + F(\theta_{a}(x)-t) -2 f(x).
\end{equation*}
In this case the last equality written as

\begin{equation*}
S^{a}_{n}(f)(x)-f(x)= \frac{1}{2\pi} \int_{0}^{\pi} \left\{ \frac {\phi^{a}_{x}(t)}{\sin \left( \frac{t}{2}\right)}   \sin  \left(\frac{2n+1}{2} \right) t \right\}dt.
\end{equation*}
From  hypothesis   $\phi^{a}_{x}(t)\in L^{1}(\mathbb{T})$. Hence

\begin{equation*}
\int_{0}^{\pi} \left|\frac{ \phi^{a}_{x}(t)}{t} \right| dt< \infty
\end{equation*}
holds and by using Jordan inequality (see  \cite{Zygmund}), we obtain

\begin{equation*}
\int_{0}^{\pi} \left| {\frac{ \phi^{a}_{x}(t)}{\sin\left( \frac{t}{2}  \right)} }\right| dt<\infty.
\end{equation*}
Therefore

\begin{equation*}
S^{a}_{n}(f)(x)-f(x)= \frac{1}{2\pi} \int_{0}^{\pi} \left\{ \frac {\phi^{a}_{x}(t)}{\sin\left( \frac{t}{2}  \right)} \right\} \sin \left( \left(
n+\frac {1}{2}\right) t \right)dt
\end{equation*}
\begin{equation*}
= \frac{1}{2\pi} \int_{0}^{\pi} \left\{ \frac {\phi^{a}_{x}(t)}{ \sin \left( \frac{t}{2}  \right)} \right\} \sin (nt) \cos \left(\frac{t}{2}\right)dt
\end{equation*}
\begin{equation*}
+\frac{1}{2\pi} \int_{0}^{\pi} \left\{ \frac {\phi^{a}_{x}(t)}{\sin\left( \frac{t}{2}  \right)} \right\} \cos (nt) \sin \left(\frac{t}{2}\right)dt
\end{equation*}
holds. Here if we say that

\begin{equation*}
g^{a}_{x}(t)=   \left\{ \frac {\phi^{a}_{x}(t)}{\sin\left( \frac{t}{2}  \right)} \right\} \cos \left(\frac{t}{2}\right).
\end{equation*}
We find that $g^{a}_{x}(t)\in L^{1}(\mathbb{T})$. In this case, considering the following equalities

\begin{equation*}
b_{n}\left( g^{a}_{x}(t)\right) =  \frac{1}{2\pi} \int_{-\pi}^{\pi}  g^{a}_{x}(t) \sin (nt)dt
\end{equation*}
\begin{equation*}
a_{n}\left( \phi^{a}_{x}(t)\right) =  \frac{1}{2\pi} \int_{-\pi}^{\pi}  \phi^{a}_{x}(t) \cos(nt)dt
\end{equation*}
we see that

\begin{equation*}
S^{a}_{n}(f)-f(x)= \frac{1}{4}\left( b_{n}\left( g^{a}_{x}(t) \right)\right) +a_{n} \left( \phi^{a}_{x}(t)\right)
\end{equation*}
holds. Thus by considering Corollary 2.2 and Riemann-Lebesgue Lemma (see in \cite{Zygmund}) for $\mid n \mid\rightarrow \infty$, we have

\begin{equation*}
b_{n}\left( g^{a}_{x}(t)\right) \rightarrow 0 \quad and \quad a_{n}\left( \phi^{a}_{x}(t)\right) \rightarrow 0.
\end{equation*}
So we obtain $S^{a}_{n}(f)(x)\rightarrow f(x)$ as  $ n \rightarrow \infty$. This completes the proof.
\end{proof}
\begin{corollary}
Let $f \in L^{1}(\mathbb{T})$. If $\mid F(x)-F(y) \mid\leq M \mid  x-y \mid^{\alpha}$, then  we have
\begin{equation*}
S^{a}_{n}(f)(x)\rightarrow f(x),\quad n\rightarrow\infty.
\end{equation*}
\end{corollary}
\begin{proof}
Let assume that $f \in L^{1}(\mathbb{T})$. Therefore

\begin{equation*}
 \left| \frac{F(\theta_{a}(x)+t) + F(\theta_{a}(x)-t) -2 f(x)}{t }\right|
\end{equation*}
\begin{equation*}
\leq \frac{\mid F(\theta_{a}(x)+t)-f(x)\mid +\mid F(\theta_{a}(x)-t) - f(x)\mid}{t }
\end{equation*}
\begin{equation*}
\leq \frac{2Mt^{\alpha}}{t}=2M t^{\alpha-1}
\end{equation*}
holds. Thus if we consider  $\int^{\pi}_{0}t^{\alpha-1} dt= \frac{\pi^{\alpha}}{\alpha}$ equality, we have the \linebreak following by using comparison test.

\begin{equation*}
\frac{1}{2\pi} \int_{0}^{\pi} \left|\frac{F(\theta_{a}(x)+t) + F(\theta_{a}(x)-t) -2 f(x)}{t}\right| dt < \infty
\end{equation*}
Therefore from Theorem 2.6, we obtain $S^{a}_{n}(f)(x)\rightarrow f(x), \quad n\rightarrow\infty$.
\end{proof}
\begin{corollary}
Let $f \in L^{1}(\mathbb{T})$. If $f$ is differentable and $F:= f\circ \theta_{a}^{-1}$ for all $x \in \mathbb{T}$, then we have

\begin{equation*}
S^{a}_{n}(f)(x)\rightarrow f(x),\quad n \rightarrow \infty.
\end{equation*}
\end{corollary}
\begin{proof}
Let $f \in L^{1}(\mathbb{T})$. Hence

\begin{equation*}
F'(\theta_{a}(x))=\lim_{t\rightarrow 0} \frac{F(\theta_{a}(x)+t)-F(\theta_{a}(x))}{t}
\end{equation*}
holds. For  $\exists \quad \delta>0 \ni 0<\mid t \mid < \delta$

\begin{equation*}
\left|\frac{F(\theta_{a}(x)+t)-F(\theta_{a}(x))}{t}-F'(\theta_{a}(x))\right|< 1
\end{equation*}
holds. Thus we have

\begin{equation*}
\frac{\left| F(\theta_{a}(x)+t)-F(\theta_{a}(x))\right|}{\mid t \mid}< \mid F'(\theta_{a}(x))\mid.
\end{equation*}
Therefore
\begin{equation*}
\frac{ \mid F(\theta_{a}(x)+t)-F(\theta_{a}(x))\mid}{\mid t \mid}  < \frac{\mid F(\theta_{a}(x)+t)-F(\theta_{a}(x))\mid}{\mid \delta \mid}
\end{equation*}
holds for $\mid t \mid >\delta$. Hence we have the following inequality.

\begin{equation*}
\int_{0}^{\pi} \left| \frac{F(\theta_{a}(x)+t) + F(\theta_{a}(x)-t) -2 f(x)}{t}\right| dt
\end{equation*}
\begin{equation*}
= \int_{0<t<\delta} \left|  \frac{F(\theta_{a}(x)+t) +F(\theta_{a}(x)-t) -2 f(x)}{t} \right|  dt
\end{equation*}
\begin{equation*}
+ \int_{\delta \leq t \leq \pi}\left| \frac{F(\theta_{a}(x)+t) + F(\theta_{a}(x)-t) -2f(x)}{t}\right|  dt
\end{equation*}
\begin{equation*}
\leq \int_{0<t<\delta} \left | \frac{F(\theta_{a}(x)+t)-f(x)}{t}\right| dt+ \int_{0<t<\delta} \left| \frac{ F(\theta_{a}(x)-t)-f(x) }{t}
\right| dt
\end{equation*}
\begin{equation*}
+ \int_{\delta \leq t \leq \pi}\left |\frac{F(\theta_{a}(x)+t) + F(\theta_{a}(x)-t) -2f(x)}{t}\right | dt
\end{equation*}
\begin{equation*}
\leq \int_{0<t<\delta} \mid t \mid F'(\theta_{a}(x))dt+ \int_{0<t<\delta} \mid t \mid F'(\theta_{a}(x))dt
\end{equation*}
\begin{equation*}
+\int_{\delta \leq t \leq \pi}\left|\frac{F(\theta_{a}(x)+t) + F(\theta_{a}(x)-t) -2f(x)}{\delta}\right|dt
\end{equation*}
So by considering  Theorem 2.6, we obtain

\begin{equation*}
S^{a}_{n}(f)(x)\rightarrow f(x),\quad n\rightarrow\infty.
\end{equation*}
\end{proof}

\section{Boundness and Ces\`{a}ro Mean  Summability for Nonlinear Fourier Series}

Precisely, we prove the necessary and sufficient condition for nonlinear Fourier series which governs the $(C,1)$ summability in $L^{p}$ for arbitrary function $f$ from $L^{p}$. This result is applied to the prove of Bernstein's inequality for nonlinear trigonometric polynomials.
Let $f \in L^{1}(\mathbb{T})$ and $S_{n}^{a}(f)$ is partial sum of nonlinear Fourier series. Define $n$-th  Fej\'{e}r (Ces\`{a}ro) for nonlinear Fourier series defined  by

\begin{equation*}
\sigma_{n}^{a}(f)(x):= \frac{1}{n+1}\sum_{k=0}^{n} S_{k}^{a}(f)(x).
\end{equation*}

Now let find new expressions for Ces\`{a}ro mean which are very useful.
\begin{proposition}
Let $\sigma_{n}^{a}(f)$ be the sequence of partial sums of \linebreak the nonlinear Fourier series of $f$. Let  $f \in L^{1}(\mathbb{T})$  which is $2\pi$-periodic. Then

\begin{equation}\label{3:1}
\sigma_{n}^{a}(f)(x)= \frac{1}{2\pi} \int_{-\pi}^{\pi} f(t) K_{n}(\theta_{a}(x)-\theta_{a}(t)) p_{a}(t)dt,
\end{equation}
\begin{equation}\label{3:2}
= \frac{1}{2\pi} \int_{-\pi}^{\pi} F(\theta_{a}(x)+t) K_{n}(t)dt,
\end{equation}
\begin{equation}\label{3:3}
= \frac{1}{2\pi} \int_{0}^{\pi} \left\{ F(\theta_{a}(x)-t)+F(\theta_{a}(x)+t)\right \} K_{n}(t)dt.
\end{equation}

\end{proposition}

\begin{proof}
 By the expression for the $c_{k}$, phase function $\theta_{a}(x)$  and \linebreak  considering $\sigma_{n}^{a}(f)(x):=\frac{1}{n+1} \sum^{n}_{k=1} \sigma ^{a}_{k}(f)(x)$, we have already \linebreak  established \eqref{3:1}.  By a change of variable $y=\theta_{a}(t)-\theta_{a}(x)$ and the \eqref{1:1} equality of   $\theta_{a}(x)$'s $2\pi$-periodicity,  it follows that the integral does not change as long as the length of the interval of integration is $2\pi$ we get
\begin{equation*}
\sigma_{n}^{a}(f)(x)= \frac{1}{2\pi} \int_{-\pi}^{\pi} F(\theta_{a}(x)+y) D_{n}(y)dy.
\end{equation*}
This proves \eqref{3:2}. Finally, we split the integral in \eqref{3:2} as the sum of integrals over $[-\pi, 0]$ and $[0, \pi]$. Now
\begin{equation*}
\int_{-\pi}^{0} F(\theta_{a}(x)+t) K_{n}(t)dt=\int_{0}^{\pi} F(\theta_{a}(x)+t) D_{n}(t)dt,
\end{equation*}
using the change of variable $y =-t$ and the evenness of $K_{n}$. This proves \eqref{3:3}.
\end{proof}

\begin{theorem}
Let $\sigma_{n}^{a}: \mathcal{C}(\mathbb{T})\rightarrow \mathcal{C}(\mathbb{T})$ for  $n=0, 1, ...$ be  linear operator sequence.  Then we have
\begin{equation*}
\parallel \sigma_{n}^{a}(f)\parallel_{\infty} \leq \parallel f \parallel_{\infty}
\end{equation*}
for all  $f \in \mathcal{C}(\mathbb{T})$.
\end{theorem}
\begin{proof}
Let consider $\sigma_{n}^{a}:\mathcal{C} (\mathbb{T})\rightarrow \mathcal{C}(\mathbb{T})$ as a linear operator  for \linebreak $n=0, 1, ...$. Thus by using \eqref{3:2} inequality and Fej\'{e}r kernel's \linebreak positivity (see in \cite{Lorentz}), we have
\begin{equation*}
\mid \sigma_{n}^{a}(f)(x)\mid =\left|\frac{1}{2\pi} \int_{-\pi}^{\pi} F(\theta _{a}(x)-t) K_{n}(t) dt \right|
\end{equation*}
\begin{equation*}
\leq \frac{1}{2\pi} \int_{-\pi}^{\pi} \mid F(\theta _{a}(x)-t)\mid \mid K_{n}(t )\mid dt,
\end{equation*}
 \begin{equation*}
\leq \frac{1}{2\pi} \int_{-\pi}^{\pi} \parallel F \parallel_{\infty}  K_{n}(t)dt.
\end{equation*}
If we consider  Lemma 2.1 for last equality, we obtain that
\begin{equation*}
\mid \sigma_{n}^{a}(f)(x)\mid= \frac{1}{2\pi} \int_{-\pi}^{\pi} \parallel f \parallel_{\infty} K_{n}(t)dt.
\end{equation*}
Finally by using Fej\'{e}r  kernel's property (see in \cite{Zygmund}) we obtain the result.
\begin{equation*}
\parallel \sigma_{n}^{a}(f)\parallel_{\infty} \leq \parallel f \parallel_{\infty}
\end{equation*}
\end{proof}
\begin{theorem}
 Let $f\in \mathcal{C} (\mathbb{T})$. Then
\begin{equation*}
\parallel f- \sigma_{n}^{a}(f) \parallel_{\infty} \rightarrow 0, \quad n\rightarrow \infty.
\end{equation*}
\end{theorem}
\begin{proof}
Let $\sigma_{n}^{a}: \mathcal{C} (\mathbb{T})\rightarrow \mathcal{C} (\mathbb{T})$, $n=0, 1, ...$. Then we have

\begin{equation*}
| f(x)-\sigma_{n}^{a}(f)(x) | =\left| \frac{1}{2\pi} \int_{-\pi}^{\pi} [f(x)- F(\theta_{a}(x)+ t)] K_{n}(t)dt\right|
\end{equation*}

\begin{equation*}
\leq \frac{1}{2\pi} \int_{\mid t \mid <\delta }\mid f(x)- F(\theta _{a}(x)+t)\mid K_{n}(t)dt
\end{equation*}
\begin{equation*}
+ \frac{1}{2\pi} \int_{\delta \leq \mid t\mid \leq \pi}\mid f(x)- F(\theta _{a}(x)+t) \mid K_{n}(t) dt
\end{equation*}
by \eqref{3:2} equality. Therefore  since $f\in C(\mathbb{T})$ and $\theta _{a}$ continuous, we can write

\begin{equation*}
| f(x)-\sigma_{n}^{a}(f)(x) |\leq \frac{1}{2\pi} \int_{\mid t\mid <\delta } \frac{\varepsilon}{2} K_{n}(t) dt
\end{equation*}
\begin{equation*}
+ \frac{1}{2\pi} \int_{\delta \leq \mid t \mid \leq \pi}\mid f(x)- F(\theta _{a}(x)+t) \mid K_{n}(t)dt
\end{equation*}
\begin{equation*}
\leq \frac{\varepsilon}{4\pi} \int_{\mid t\mid <\delta } K_{n}(t)dt+ \frac{1}{2\pi} \int_{\delta \leq \mid t \mid \leq \pi}\mid f(x)\mid + \mid
F(\theta _{a}(x)+t) \mid K_{n}(t)dt
\end{equation*}

\begin{equation*}
\leq \frac{\varepsilon}{4\pi} +\frac{1}{2\pi} \int_{\delta \leq \mid t\mid \leq \pi} [\parallel f\parallel_{\infty} + \parallel
F\parallel_{\infty} ]K_{n}(t) dt.
\end{equation*}
Now if we use Lemma 2.1 and Fej\'{e}r  kernel's property (see in \cite{Zygmund}), we have

\begin{equation*}
\mid f(x)-\sigma_{n}^{a}(f)(x) \mid \leq \frac{\varepsilon}{4\pi} +\frac{1}{2\pi} 2 \parallel f \parallel_{\infty} \int_{\delta \leq \mid t\mid \leq \pi}
K_{n}(t)dt
\end{equation*}

\begin{equation*}
\leq \frac{\varepsilon}{4\pi} +\frac{1}{\pi} \parallel f \parallel_{\infty} \int_{\delta \leq \mid t\mid \leq \pi}
\frac{1}{n+1}\frac{1}{sin^{2}\frac{\delta}{2}} dt
\end{equation*}

\begin{equation*}
\leq \frac{\varepsilon}{2} + \frac{\parallel f \parallel_{\infty}}{\pi(n+1)sin^{2}\frac{\delta}{2}} \int_{\delta \leq \mid t\mid\leq \pi} dt
\end{equation*}

\begin{equation*}
\leq \frac{\varepsilon}{2} + \frac{2 \parallel f \parallel_{\infty}}{\pi(n+1)sin^{2}\frac{\delta}{2}}.
\end{equation*}
Since there exists  $\exists N \in \mathbb{N}:\quad n \geq N $
\begin{equation*}
\frac{2 \parallel f \parallel_{\infty}}{\pi(n+1)sin^{2}\frac{\delta}{2}}\rightarrow 0
\end{equation*}
as $n \rightarrow \infty $, we obtain
\begin{equation*}
\frac{2 \parallel f \parallel_{\infty}}{\pi(n+1)sin^{2}\frac{\delta}{2}}< \frac{\varepsilon}{2}.
\end{equation*}
Therefore
\begin{equation*}
\mid f(x)- \sigma_{n}^{a}(f)(x)\mid < \frac{\varepsilon}{2}+\frac{\varepsilon}{2}= \varepsilon
\end{equation*}
holds for $\exists N \in \mathbb{N} : n \geq N $ ve $\forall x\in \mathbb{T}$. Thus $\sigma_{n}^{a}(f)\rightarrow f$ uniformly as $n \rightarrow \infty$
\end{proof}

\begin{theorem}
Let consider  $\sigma_{n}^{a}: L^{p}(\mathbb{T})\rightarrow L^{p}(\mathbb{T})$, $1\leq p<\infty$ and $n=0, 1, ...$ linear operator sequence. Then
\begin{equation*}
\parallel \sigma_{n}^{a}(f)\parallel_{p} \leq  \left(\frac{1+\left\vert a\right\vert }{1-\left\vert a\right\vert }\right)^{2/p} \parallel f
\parallel_{p}
\end{equation*}
holds for all $f \in L^{p}(\mathbb{T})$.
\end{theorem}

\begin{proof}
Let  consider $\sigma_{n}^{a}: L^{p}(\mathbb{T})\rightarrow L^{p}(\mathbb{T})$,$1\leq p<\infty$ and  $n=0, 1, ...$ linear operator sequence. Then we have
\begin{equation*}
\parallel \sigma_{n}^{a}(f)\parallel_{p}= \left(\frac{1}{2\pi}\int_{-\pi}^{\pi}\mid \sigma_{n}^{a}(f)(x)\mid^{p} dx\right)^{1/p}.
\end{equation*}
Hence by \eqref{3:2} equality, we have
\begin{equation*}
\parallel \sigma_{n}^{a}(f)\parallel_{p}=\frac{1}{2\pi}\int_
{-\pi}^{\pi} \left(\frac{1}{2\pi} \int_{-\pi}^{\pi} \mid
F(\theta _{a}(x)- t) K_{n}(t)
dt\mid^{p} dx \right)^{1/p}.
\end{equation*}
Now by using  Minkowski integral inequality (see in \cite{Zygmund}),
\begin{equation*}
\parallel \sigma_{n}^{a}(f)\parallel_{p}=\frac{1}{2\pi}\int_{-\pi}^{\pi} \left(\frac{1}{2\pi} \int_{-\pi}^{\pi} \mid F(\theta _{a}(x)- t)\mid^{p} dx
\right)^{1/p} K_{n}(t) dt
\end{equation*}
\begin{equation*}
=\frac{1}{2\pi}\int_{-\pi}^{\pi} \parallel F \parallel_{p} K_{n}(t) dt
\end{equation*}
holds. Therefore we obtain the following result by using Lemma 2.1.
\begin{equation*}
\parallel \sigma_{n}^{a}(f)\parallel_{p} \leq \left(\frac{1+\left\vert a\right\vert }{1-\left\vert a\right\vert }\right)^{2/p}\parallel f \parallel_{p}
\end{equation*}
\end{proof}

\begin{theorem}
Let $1\leq p < \infty$.Then we have
\begin{equation*}
\parallel f- \sigma_{n}^{a}(f) \parallel_{p} \rightarrow 0,\quad  n\rightarrow \infty
\end{equation*}
for all  $f \in L^{p}(\mathbb{T})$.
\end{theorem}

\begin{proof}
Let $\sigma_{n}^{a}: L_{p}(\mathbb{T})\rightarrow L_{p}(\mathbb{T})$,  $n=0, 1,... $ linear operator  and $f \in L_{p}(\mathbb{T})$.  Take $\varepsilon > 0$.Then there exists   $\exists \quad g\in \mathcal{C}(\mathbb{T})$ such that

\begin{equation*}
\parallel f- g \parallel_{p}< \frac{\varepsilon}{3}.
\end{equation*}

Let  $g \in \mathcal{C}(\mathbb{T})$. Then for $n\rightarrow\infty$, we have $\sigma_{n}^{a}(g)\rightarrow g $ converges uniformly by Theorem 3.2.
Thus by using the Lusin Theorem (see in \cite{Lorentz}), there exists $\exists \quad n \geq N$ such that $ \parallel \sigma_{n}^{a}(g)- g \parallel_{p}<\frac{\varepsilon}{3}$.
\begin{equation*}
\parallel f- \sigma_{n}^{a}(f)\parallel_{p}= \parallel f-g+g-\sigma_{n}^{a}(g)+\sigma_{n}^{a}(g)- \sigma_{n}^{a}(f)\parallel_{p}
\end{equation*}

\begin{equation*}
\leq \parallel f- g \parallel_{p}+ \parallel g- \sigma_{n}^{a}(g) \parallel_{p}+ \parallel \sigma_{n}^{a} (f- g) \parallel_{p}
\end{equation*}
Therefore by Theorem 3.3, we have
\begin{equation*}
\parallel f- \sigma_{n}^{a}(f)\parallel_{p}\leq \parallel f- g \parallel_{p}+ \parallel g- \sigma_{n}^{a}(g)
\parallel_{p}+ \left(\frac{1+\left\vert a\right\vert }{1-\left\vert a\right\vert }\right)^{1/p} \parallel f- g \parallel_{p}
\end{equation*}
\begin{equation*}
\leq \left(\left(\frac{1+\left\vert a\right\vert }{1-\left\vert a\right\vert }\right)^{2/p}+1\right)\parallel f- g
\parallel_{p}+\frac{\varepsilon}{3}< \varepsilon.
\end{equation*}
This completes the proof.
\end{proof}

\section{Bernstein's inequality for Nonlinear Fourier Series}

Applying the inequalities for the Ces\`{a}ro means of nonlinear \linebreak trigonometric series derived in the previous section, we can prove the nonlinear version of the well-known Bernstein's inequality. For any trigonometric polynomial $t_{n}(x)$ of order $\leq n$, for every $1\leq p \leq \infty$,
we have
 \begin{equation}\label{4:1}
\left( \int_{-\pi}^{\pi}\mid t'_{n}(x)\mid^{p} dx \right)^{1/p} \leq c  n \left( \int_{-\pi}^{\pi}\mid t_{n}(x)\mid^{p} dx \right)^{1/p}.
\end{equation}
The last inequality is known as integral Bernstein's inequality.
The following extension of \eqref{4:1} is true.
\begin{theorem}
Let $1 < p <\infty$ and assume that $t_{a} \in \tau_{n}^{a}$. Then the  inequality
\begin{equation}
\int_{-\pi}^{\pi}\mid t'_{n,a}(x)\mid^{p} dx \leq C \left( \frac{1+\mid a \mid}{ 1-\mid a \mid}\right)^{2 \left(1+1/p\right)}  \int_{-\pi}^{\pi}\mid t_{n,a}(x)\mid^{p} dx
\end{equation}
holds.
\end{theorem}
\begin{proof}
It is well known from the \eqref{2:1} equality that
\begin{equation*}
\frac{1}{\pi}\int_{-\pi}^{\pi}\mid t_{n,a}(u)\mid^{p} D_{n}(\theta_{a}(u)-\theta_{a}(x))p_{a}(u) du
\end{equation*}
where
\begin{equation*}
D_{n}(u)=\frac{1}{2}+ \sum_{k=1}^{n} \cos ku
\end{equation*}
is the Dirichlet's kernel of order $n$. Let $T= t_{n,a}\circ \theta_{a}^{-1}$. By the derivation, we obtain
\begin{equation*}
t'_{n,a}(x)=\frac{1}{\pi}\int_{-\pi}^{\pi}\mid t_{n,a}(u)\mid^{p} D'_{n}(\theta_{a}(u)-\theta_{a}(x))(-p_{a}(x))p_{a}(u) du
\end{equation*}
\begin{equation*}
=p_{a}(x)\frac{1}{\pi}\int_{-\pi}^{\pi} T_{n,a}(\theta_{a}(u)) D'_{n}(\theta_{a}(u)-\theta_{a}(x))p_{a}(u) du
\end{equation*}
\begin{equation*}
=p_{a}(x)\frac{1}{\pi}\int_{-\pi}^{\pi} T_{n,a}(y+\theta_{a}(x)) D'_{n}(y)dy
\end{equation*}
\begin{equation*}
=p_{a}(x)\frac{1}{\pi}\int_{-\pi}^{\pi} T_{n,a}(\theta_{a}(x)+y)\left( \sum_{k=1}^{n}  k \sin ky  \right)dy
\end{equation*}
\begin{equation*}
=p_{a}(x)\frac{1}{\pi}\int_{-\pi}^{\pi} T_{n,a}(\theta_{a}(x)+y) \left( \sum_{k=1}^{n}  (k \sin ky)+ \sum_{k=1}^{n-1} (k \sin (2n-k)y) \right)dy
\end{equation*}
\begin{equation*}
=p_{a}(x)\frac{1}{\pi}\int_{-\pi}^{\pi} T_{n,a}(\theta_{a}(x)+y) 2n \sin ny \left( \frac{1}{2}+ \sum_{k=1}^{n-1} \left(1-\frac{k}{n} \cos ky \right) \right)dy
\end{equation*}
\begin{equation*}
=p_{a}(x)\frac{1}{\pi}\int_{-\pi}^{\pi} T_{n,a}(\theta_{a}(x)+y) 2n \sin ny K_{n-1}(y)dy
\end{equation*}
where $K_{n-1}$ is the Fej\'{e}r's kernel of order $n-1$. By taking the absolute values, we get
\begin{equation*}
\mid T'_{n,a} \mid \leq  \left( \frac{1+\mid a \mid}{1-\mid a \mid}\right)\mid 2n \sin ny  \mid \frac{1}{\pi}\int_{-\pi}^{\pi} T_{n,a}(\theta_{a}(x)+y) K_{n-1}(y)dy
\end{equation*}
\begin{equation*}
\leq  \left( \frac{1+\mid a \mid}{ 1-\mid a \mid}\right) 2n \frac{1}{\pi}\int_{-\pi}^{\pi} T_{n,a}(\theta_{a}(x)+y) K_{n-1}(y)dy
\end{equation*}
\begin{equation*}
\leq  \left( \frac{1+\mid a \mid}{ 1-\mid a \mid}\right) 2n  \sigma^{a}_{n-1}(t_{n,a},x).
\end{equation*}
If we use Theorem 3.3, we get that
\begin{equation*}
\frac{1}{2 \pi}\left( \int_{-\pi}^{\pi}\mid t'_{a}(x)\mid^{p} dx \right)^{1/p} \leq \frac{1}{2 \pi}\left( \int_{-\pi}^{\pi}\mid \left( \frac{1+\mid a \mid}{ 1-\mid a \mid}\right) 2n  \sigma^{a}_{n-1}(t_{n,a},x)\mid^{p} dx \right)^{1/p}
\end{equation*}

\begin{equation*}
\leq \left( \frac{1+\mid a \mid}{ 1-\mid a \mid} \right)  2n  \left(\frac{1}{2 \pi} \int_{-\pi}^{\pi}\mid\sigma^{a}_{n-1}(t_{n,a}, x) \mid^{p} dx \right)^{1/p}
\end{equation*}

\begin{equation*}
\leq \left( \frac{1+\mid a \mid}{ 1-\mid a \mid}\right)  2n \sigma^{a}_{n-1}(t_{n,a}, x)\parallel_{p}
\end{equation*}
\begin{equation*}
\leq \left( \frac{1+\mid a \mid}{ 1-\mid a \mid}\right)  2n \left( \frac{1+\mid a \mid}{ 1-\mid a \mid}\right)^{2/p} \parallel t_{n,a} \parallel_{p}
\end{equation*}

\begin{equation*}
\leq   2n \left( \frac{1+\mid a \mid}{ 1-\mid a \mid}\right)^{2+1/p} \parallel t_{n,a} \parallel_{p}
\end{equation*}
From this we obtain
\begin{equation*}
\parallel t'_{n,a} \parallel_{p}\leq   c n \left( \frac{1+\mid a \mid}{ 1-\mid a \mid}\right)^{2+1/p} \parallel t_{a} \parallel_{p}
\end{equation*}
and this completes the proof.
\end{proof}

\end{document}